\documentclass[12pt,reqno,a4paper]{amsart}
\usepackage{amsmath,amssymb,amsthm,amsfonts}
\usepackage[dvips]{graphicx}

\newtheorem{theorem}{Theorem}
\newtheorem{lemma}[theorem]{Lemma}
\newtheorem{proposition}[theorem]{Proposition}
\newtheorem{corollary}[theorem]{Corollary}

\theoremstyle{definition}
\newtheorem{defn}[theorem]{Definition}
\theoremstyle{remark}
\newtheorem{remark}[theorem]{Remark}

\newcommand{\Rr}{\mathbb{R}}
\newcommand{\Nn}{\mathbb{N}}
\newcommand{\Zz}{\mathbb{Z}}

\newcommand{\I}{\mathcal{I}}

\DeclareMathOperator{\Leb}{Leb}

\DeclareMathOperator{\diam}{diam}

\def\veca{{\text{\boldmath$a$}}}
\def\vecb{{\text{\boldmath$b$}}}
\def\vecx{{\text{\boldmath$x$}}}
\def\vecy{{\text{\boldmath$y$}}}

\usepackage[colorlinks, citecolor=black, urlcolor=black, linkcolor=black]{hyperref}
\usepackage{url}
\usepackage{color}
\usepackage{tikz}
\usetikzlibrary{shapes, arrows}
\usepackage{pgfplots}

\usepackage{graphics}
\usepackage{amsmath}
\usepackage{amsthm}
\usepackage{amssymb}
\usepackage{amscd}
\usepackage{amsfonts}
\usepackage{amsbsy}
\usepackage{epsfig}
\usepackage{color}

\begin{document}

\title[Dynamics of  piecewise increasing contractions]
{Dynamics of  piecewise increasing  contractions}

\author[Gaiv\~ao]{Jos\'e Pedro Gaiv\~ao}
\address{Departamento de Matem\'atica and CEMAPRE/REM, ISEG\\
Universidade de Lisboa\\
Rua do Quelhas 6, 1200-781 Lisboa, Portugal}
\email{jpgaivao@iseg.ulisboa.pt}

\author[Nogueira]{Arnaldo Nogueira}

\address{Aix Marseille Universit\'e, CNRS\\ Centrale Marseille, Institut de Math\'ematiques de Marseille\\ 163 avenue de Luminy, Case 907, 13288 Marseille, Cedex 9, France}
\email{arnaldo.nogueira@univ-amu.fr}

\date{\today}

\begin{abstract}

Let  $I_1=[a_0,a_1),\ldots,I_{k}= [a_{k-1},a_k)$ be a partition of the interval $I=[0,1)$ into  $k$ subintervals. Let
 $f:I\to I$ be a map  such that each restriction $f|_{I_i}$ is an increasing   Lipschitz contraction.
 We prove that any   $f$ admits at most  $k$  periodic orbits, where the upper bound is sharp. 
 We are also interested in the dynamics of piecewise linear $\lambda$-affine maps, where $0<\lambda<1$. Let $b_1,\ldots,b_k$ be real numbers and 
 let  $F_\lambda: I\to \Rr$ be a function such that each restriction $F_\lambda|_{I_i}(x)=\lambda x +b_i$. 
 Under a generic assumption on the parameters $a_1,\ldots,a_{k-1},b_1,\ldots,b_k$, we prove that, up to a zero Hausdorff dimension set of slopes $\lambda$, the $\omega$-limit set of the piecewise $\lambda$-affine maps $f_\lambda:x\in I \mapsto F_\lambda(x)\pmod{1}$ at every point equals a periodic orbit and there exist at most $k$ periodic orbits.
Moreover, let $\mathfrak{E}^{(k)}$ be the exceptional  set of parameters $\lambda,a_1,\ldots,a_{k-1},b_1,\ldots,b_k$  which define non-asymptotically periodic $f$, we prove that 
 $\mathfrak{E}^{(k)}$ is a Lebesgue null measure set whose Hausdorff dimension is large or equal to $k$.
\end{abstract}

\maketitle

\section{Introduction}\label{sec:intro}

\noindent
Let $I=[0,1)$ and $f:I\to I$ be an interval map which is continuous up to finitely
many points and right continuous at every discontinuity point. We call $f$ a \textit{piecewise increasing contraction}, if there exists 
$0<\lambda<1$ such that on every domain $D$ of continuity of $f$, the restriction map $f|_D$ is increasing and $\lambda$-Lipschitz\footnote{$|f(x)-f(y)| \leq \lambda |x-y|$ holds for any $x,y\in D$.}. Throughout the paper \textit{increasing} means strictly increasing.

 Let $x\in I$ and denote by $\omega(f,x)$ the $\omega$-limit set of $f$ at the point $x$ and 
$$
\displaystyle \omega(f)= \bigcup_{x\in I} \omega(f,x).
$$
 We say that $f$ is \textit{asymptotically periodic} if, for every $x\in I$, $\omega(f,x)$ equals a periodic orbit  and $\omega(f)$ consists of finitely many periodic orbits.

The study of the dynamics of interval piecewise  contractions has  attracted the attention of many authors, in particular see  \cite{BS20,Bu93,BKLN20, CCG21, CGM20, JO19,LN18,LN19, N18, NP15,NPR18}. 
The motivation of our first theorem comes mainly from  \cite{NPR18} which shows that generically piecewise contractions are asymptotically periodic. We recall that in \cite{NP15}, the authors prove  that any injective interval piecewise contraction which has $n$ discontinuities admits at most $n+1$ periodic orbits and this upper bound is sharp. 
In this paper, our first goal is to present classes of piecewise increasing contractions, not necessarily injective,  which are asymptotically periodic and to prove an upper bound for their number of periodic orbits.

\begin{theorem}\label{th:numbermain}
Let $f$ be a piecewise increasing contraction with $n$ discontinuity points. Then $f$ has at most $n+1-\ell$ periodic orbits where $\ell$ is the number of discontinuity points whose image under $f$ equals zero.
\end{theorem}

We call attention that our approach to prove  Theorem~\ref{th:numbermain} is elementary and $f$ is not assumed to be asymptotically periodic.
We also recall that the result obtained in \cite{NPR18} for an upper bound of the number of periodic orbits concerns generic piecewise contractions, thus it can not be applied to prove Theorem~\ref{th:numbermain}.  In \cite[Theorem 1.1]{CCG21}, the upper bound $n+1$ is  obtained, however under an extra hypothesis: it is supposed that $f$ has no singular connection  (see Definition~\ref{def:sing}) which excludes the existence of an orbit which contains two discontinuity points.

Next we state a couple of corollaries of Theorem~\ref{th:numbermain}. First, Corollary~\ref{cor:circle} that generalizes previous results and also  \cite[Theorem 10]{LN18} to any number of discontinuities, and Corollary~\ref{cor:nPC} which improves the upper bound obtained  in \cite[Theorem 1.1]{NPR18} in the case of positive slope. Theorem~\ref{th:numbermain} and its corollaries are proved in Section~\ref{sec:preliminary}.

Identifying the circle $\Rr/\Zz$ with the interval $I$ through the canonical bijection $I\hookrightarrow \Rr \to \Rr/\Zz$, we may see any 
 orientation-preserving piecewise contraction circle map as a piecewise increasing contraction. 

\begin{corollary}\label{cor:circle}
Let $f$ be a  circle map which is  an orientation-preserving piecewise contraction.  Assume that $f$ has $k$ points of discontinuity on the circle $\Rr/\Zz$ and $f$ is right continuous at those points. Then $f$ has at most $k$ periodic orbits and this upper bound is sharp.
\end{corollary}

Here we are also interested in the dynamics of  piecewise $\lambda$-affine maps. Throughout the article $\Nn$ denotes the set of positive integers $\{1,2,\ldots\}$. Given $k\in \Nn$, let 
$$
A^{(k)}=\{(0,a_1, \ldots,a_{k-1})\in \Rr^{k}: 0<a_1<\ldots<a_{k-1}<1\}.
$$

Let $0<\lambda<1$.  A function $F: I \to \Rr$ is called a {\it $k$-interval piecewise $\lambda$-affine function}, if there exist $\veca=(0,a_1, \ldots,a_{k-1})\in A^{(k)}$ and $\vecb=(b_1,\ldots,b_k)\in \Rr^k$ such that 
\begin{equation}\label{eq:FF}
F(x)= \lambda x+ b_i, \; \forall \;1\leq i\leq k \;\mbox{and} \; \forall \; x\in[a_{i-1},a_i),
\end{equation}
where $a_0=0$ and $a_k=1$.
We call   $f\colon I \to I $ defined by $f(x)=F(x)\pmod{1}$ a {\it piecewise $\lambda$-affine map}. For this special case, we prove the following result.

\begin{corollary}\label{cor:nPC}
Let $0<\lambda<1$ and let $F\colon I \to \Rr$ be a $k$-interval piecewise $\lambda$-affine function \eqref{eq:FF}. Then, the map $f = F \pmod{1}$ has at most $k$ periodic orbits.
\end{corollary}

In \cite{NP15}, this result was obtained under the assumptions that the map is injective.
In \cite{NPR18}, the authors considered a family of piecewise $\lambda$-affine maps, where $0<\vert \lambda \vert <1$ is fixed. Precisely, let $0<\vert \lambda \vert <1$ and $F$ be a $k$-interval piecewise $\lambda$-affine function given by \eqref{eq:FF}.
Let  $\delta $ be a real parameter and define a parametrized family of piecewise $\lambda$-affine maps  $f_\delta:I\to I$ using the following set-up:
$$
f_\delta:x\in I \mapsto F(x)+\delta\pmod{1}.
$$
It is proved in \cite{NPR18} that, for Lebesgue almost every $\delta$, the map $f_\delta$ is asymptotically periodic and has at most $k+1$ periodic orbits. However, the upper bound $k+1$ can be achieved only when  the slope $\lambda$ takes a negative value.

In this paper  another set-up  will be analysed.  Any piecewise $\lambda$-affine map $f$ is determined (not uniquely) by the parameters $\veca$, $\vecb$ and $\lambda$ as described above.  Henceforth, we will denote such maps by $f_{\veca,
\vecb,\lambda}$ whenever we need to stress that dependency.  Given $k\in\Nn$,  let
$$
\mathfrak{E}^{(k)}= \{(\veca,\vecb,\lambda)\in A^{(k)}\times\Rr^k\times (0,1)\colon f_{\veca,\vecb,\lambda} \text{ is not asymptotically periodic}\}.
$$

Our goal is to measure the set of slopes  $0<\lambda<1$ such that the associated  piecewise $\lambda$-affine map $f$ is not asymptotically periodic.  In other words,  given $(\veca,\vecb)\in A^{(k)}\times\Rr^k$ we wish to measure the size of the sections 
$$
\mathfrak{E}^{(k)}_{\veca,\vecb}=\{\lambda\in(0,1)\colon (\veca,\vecb,\lambda)\in \mathfrak{E}^{(k)}\}.
$$

\begin{defn}\label{def:Zind}
We say that $\vecx=(x_1,\ldots,x_{k})\in\Rr^k$ and $\vecy=(y_1,\ldots,y_k)\in\Rr^{k}$ are \textit{$\Zz$-independent}, if 
\begin{equation}\label{assumption}
x_i-y_{j}\notin \Zz,\quad \forall\, (i,j)\in\{1,\ldots,k\}^2.
\end{equation}

 \end{defn}

Notice that the condition \eqref{assumption} is a generic property.


 Our main theorem is the following:

\medskip

\begin{theorem}\label{th:main}
Let $k\in\Nn$. If $\veca\in A^{(k)}$ and $\vecb\in\Rr^k$ are $\Zz$-independent,  then $\mathfrak{E}^{(k)}_{\veca,\vecb}$ has zero Hausdorff dimension, i.e., 
$
\dim_H \mathfrak{E}^{(k)}_{\veca,\vecb}  =0.
$
\end{theorem}

Concerning the whole exceptional set $\mathfrak{E}^{(k)}$, we have the following result where Lebesgue measure means the $2k$-dimensional Lebesgue measure as we make the trivial identification of $A^{(k)}$ with a subset of $\Rr^{k-1}$ by dropping the first coordinate.

\begin{corollary}\label{cor:main}
For any $k\in\Nn$,  the exceptional set $\mathfrak{E}^{(k)}$ is a Lebesgue null measure set whose Hausdorff dimension is large or equal to $k$.
\end{corollary}

Theorem~\ref{th:main} and Corollary~\ref{cor:main} will be proved in Section~\ref{sec:proof}.

\begin{figure}[h]
\centering
\includegraphics[scale=0.6]{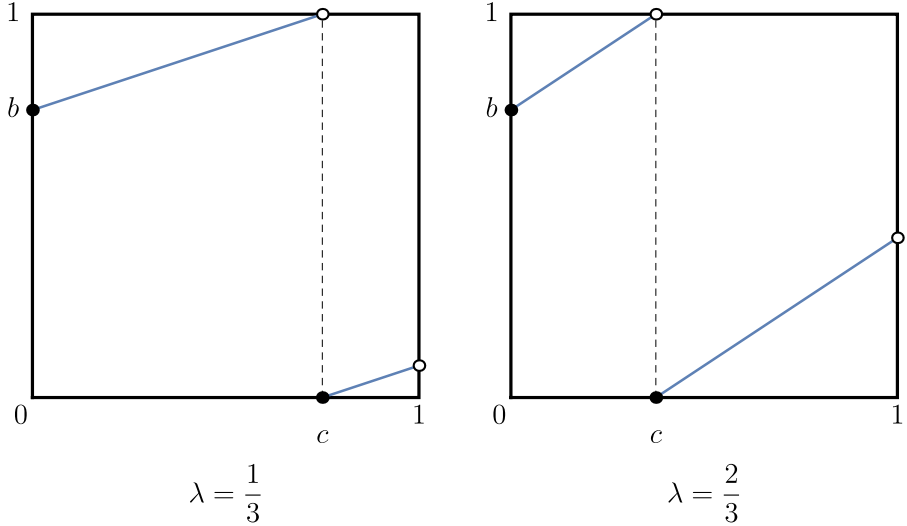}
\caption{Plot of $R_{\lambda,b}$ for two distinct values of $\lambda$.}
\label{fig:R}
\end{figure}

A concrete example to which our result applies is to the so-called contracted rotations (see Figure~\ref{fig:R}). According to \cite{LN18},  let $\lambda$ and $b$ be two real numbers such that $0<1-\lambda < b<1$ and set $c = {1-b\over \lambda}$.  A contracted rotation is the map  $R_{\lambda, b}: I \to I$  given by the splitted formula:
\begin{equation}\label{eq:contractedrotation}
R_{\lambda, b}(x) =  \begin{cases} \lambda x + b,&   \text{if} \quad 0\le x < c, 
\\
\lambda x + b -1,& \text{if}  \quad c \le x <1.
\end{cases}
\end{equation}
Notice that $R_{\lambda,b}(x)=\lambda x+ b \pmod{1}$ for every $x\in I$.  Therefore, the contracted rotation is a piecewise $\lambda$-affine map which is defined  by the $1$-interval piecewise $\lambda$-affine function $x\mapsto \lambda x +b$ with the
$\Zz$-independent $1$-tuples $\veca=(0)$ and $\vecb=(b)$.

\begin{figure}[h]
\centering
\includegraphics[scale=1]{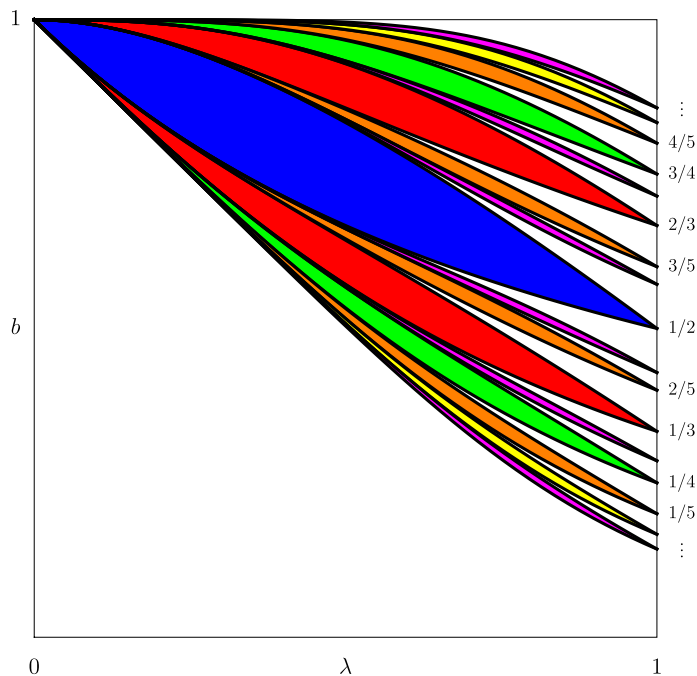}
\caption{Rational tongues: regions where $\rho(R_{\lambda,b})$ takes a rational value in the interval $(0,1)$.}
\label{fig:tongues}
\end{figure}

As we know (see \cite{Bu93,LN18}), any contracted rotation $R_{\lambda,b}$ admits a rotation number $0< \rho(R_{\lambda,b})<1$. 
 In Figure~\ref{fig:tongues}, it is described the regions in the triangle $\Delta$ formed by the parameters $(\lambda,b)$, i.e.,
$$
\Delta=\{(\lambda,b):0<1-\lambda<b<1\},
$$ 
 where the rotation number $\rho(R_{\lambda,b})$ takes a precise rational value.

The proofs of prior results about the measure and the  Hausdorff dimension of the  exceptional set 
\begin{equation}\label{exceptional1}
\mathcal{E}=\{(\lambda,b)\in \Delta\colon \rho(R_{\lambda,b}) \; \mbox{is irrational}\},
\end{equation}
use  the following proposition which describes the relation between the parameters $\lambda$ and $b$ whenever the rotation number takes a rational value (see \cite{DH87,Bu93,LN18}):

\begin{proposition}\label{pr:exc}
Let $0<\lambda<1$  and  $1\le p<q$ be relatively prime, then the rotation number of the map $R_{\lambda,b}$ takes the rational value $\displaystyle{p\over q} $ if, and only if, $b$ belongs to the interval
$$
\frac{1-\lambda}{1-\lambda^q} S\left(\lambda,\frac{p}{q}\right) \le b \le \frac{1-\lambda}{1-\lambda^q} \left( S\left(\lambda,\frac{p}{q}\right) + \lambda^{q-1}-\lambda^q \right),
$$
where $\displaystyle S\left(\lambda,\frac{1}{2}\right) =  1$ and 
$\displaystyle S\left(\lambda,\frac{p}{q}\right) =  1+\sum_{k=1}^{q-2}  \left( \left[ (k+1) \frac{p}{q} \right] -  \left[ k \frac{p}{q} \right] \right)\lambda^{k}$
 when $q>2$.
\end{proposition}

It is proved in \cite{LN18} that, once $0<\lambda<1$ is fixed, the Hausdorff dimension of the set 
$$
\{b\in(0,1) : (\lambda,b)\in \mathcal{E} \}
$$
equals  zero and,  in \cite{JO19}, the authors showed that the Hausdorff dimension of the entirely exceptional set $\mathcal{E}$ equals one. 

Concerning Proposition ~\ref{pr:exc}, according to Figure 2 a similar claim holds when it is assumed that the parameter $0<b<1$ is fixed and the rotation number takes a rational value.  However, the precise definition of the endpoints of the interval where the parameter $\lambda$ varies is not yet available.  This information would allow one  to pursue the study of the diophantine properties of the rotation number of contracted rotations.
Nevertheless,  as a corollary of our main theorem,  which means using another approach, we are able to obtain the following result:

\begin{corollary}
Let $0<b<1$, then the set
$$
\{\lambda\in(0,1) : (\lambda,b)\in\mathcal{E} \}
$$
has Hausdorff dimension zero. 
\end{corollary}

\bigskip

The paper is organized as follows. In Section~\ref{sec:preliminary}, we present some notions and  preliminary results which will be needed throughout the paper.  In particular, we show in Proposition~\ref{lem:singent} that any  piecewise increasing contraction has zero entropy. As a corollary, we obtain that the $\omega$-limit set of any piecewise increasing contraction has zero Hausdorff dimension.
Section~\ref{sec:preliminary} concludes with  the proofs of Theorem~\ref{th:numbermain} and Corollaries~\ref{cor:circle} and~\ref{cor:nPC}.
In Section~\ref{sec:partition}, we recall the notion of invariant quasi-partition and, in Theorem~\ref{thm:partition}, we establish a sufficient condition for a   piecewise increasing contraction to be asymptotically periodic. 
Section~\ref{sec:affine} is devoted to piecewise $\lambda$-affine maps, with $0<\lambda<1$.
Finally the proof of Theorem~\ref{th:main} and Corollary~\ref{cor:main} is left to Section~\ref{sec:proof}.

\section{Preliminary results}\label{sec:preliminary}

\noindent

Throughout this section $f$ is assumed to be a piecewise increasing contraction as defined in Section~\ref{sec:intro}.

\begin{defn}
A point $x\in I$ is called a \textit{singular point of $f$} if either $x=0$ or $x$ is a discontinuity point of $f$. Given $n\geq1$, a point $x\in I$ is called a \textit{singular point of $f^n$} if there exists $0\leq j<n$ such that $f^j(x)$ is a singular point of $f$. 
We denote by $S$ the set of singular points of $f$ and by $S^{(n)}$ the set of singular points of $f^n$. 
\end{defn}

Notice that $S=S^{(1)}$ and 
$$
S^{(n)}=\bigcup_{j=0}^{n-1}f^{-j}(S),\quad \forall\,n\in\Nn.
$$

\begin{lemma}\label{lem:Sn2}
Let $n\in\Nn$. Then the set $S^{(n)}$ is finite, $f^n$ is right continuous and the restriction of $f^n$ to each connected component of $I\setminus S^{(n)}$ is increasing and $\lambda^n$-Lipschitz.
\end{lemma}

\begin{proof}
Notice that $S$ is finite. Because $f$ has finitely many increasing branches, the pre-image of any finite set is also finite,  therefore $f^{-j}(S)$ is finite for every $j\geq0$. Thus, $S^{(n)}$ is finite. Now, let $J$ be a connected component of $I\setminus S^{(n)}$. The sets $J, f(J), \ldots, f^{n-1}(J)$ do not contain singular points of $f$, otherwise $J$ would not be a connected component of $I\setminus S^{(n)}$. Hence, for every $0\leq j< n$, $f^j(J)$ is an interval and $f|_{f^j(J)}$ is increasing and $\lambda$-Lipschitz. This shows that $f^n|_{J}$ is increasing and $\lambda^n$-Lipschitz. Right continuity of $f^n$ is obvious. 
\end{proof}

Since $S$ is finite, we can write $S=\{s_0,s_1,\ldots, s_{N-1}\}$ where $s_0=0<s_1<\cdots<s_{N-1}<1$. Let $s_N=1$ and consider the partition of $I$,
$$
I_1=[s_0,s_1),I_2=[s_1,s_2),\ldots, I_{N}=[s_{N-1},s_N).
$$ 
On each interval $I_j$, where $1\le j\leq N$, the restriction map $f|_{I_j}$ is increasing and $\lambda$-Lipschitz. Moreover, $f$ is right continuous at singular points.

Given  $x\in I$,  we denote the corresponding \textit{itinerary} of $x$ under $f$ on the partition $\{I_j\}_{j=1}^N$ of the interval $I$
by $(i_n)_{n=0}^\infty\in\{1,\ldots,N\}^{\Nn_0}$, where $\Nn_0=\Nn\cup\{0\}$. 
It means that $f^n(x)\in I_{i_n}$ for every $n\geq0$. A tuple $(i_0,i_1,\ldots,i_{n-1})\in\{1,\ldots,N\}^{n}$ is called an \textit{itinerary of order $n$ of $f$} if it equals the first $n$ entries of an itinerary of some point $x\in I$. 
We denote by $\I_n$ the set of all itineraries of order $n$ of $f$.  Notice that,  any two points belonging to the same connected component of $I\setminus S^{(n)}$ share the same itinerary of order $n$. Indeed,  if $J$ is a subset of a connected component of $I\setminus S^{(n)}$, then as in the proof of Lemma~\ref{lem:Sn2},  we know that the sets $J,f(J),\ldots, f^{n-1}(J)$ do not contain singular points of $f$. Thus, all points in $J$ have the same itinerary of order $n$. 

Following \cite{KM06}, the \textit{singular entropy} of $f$ is defined as
$$
h_{\text{sing}}(f)=\limsup_{n\to\infty}\frac1n\log \# \I_n,
$$
where $\#  \I_n$ denotes the cardinality of the set $ \I_n$.
By a general result \cite[Theorem 2]{KM06}, we know that the singular entropy of any non-expanding conformal piecewise affine map of $\mathbb{R}^d$  equals $0$. In our present situation, we can write a rather simple proof that $h_{\text{sing}}(f)=0$ using the fact that $f$ is an interval map. 

\begin{proposition}\label{lem:singent}
The singular entropy of $f$ equals zero, i.e. $h_{\text{sing}}(f)=0$.
\end{proposition}

\begin{proof}
Given $\rho>1$, let $m=\lceil\log 2 / \log \rho\rceil$ and $\tau=\tau(m)>0$ be the smallest distance between any two singular points of $f^m$.  We claim that the number of local distinct itineraries of order $m$ of $f$ is at most two, i.e.,  the points in any interval $J$ whose length is less than $\tau$ define at most two distinct itineraries of order $m$ of $f$.   Indeed, if $J$ is a subset of a connected component of $I\setminus S^{(m)}$, then all points in $J$ share the same itinerary of order $m$. Otherwise, if $J$ contains a singular point $c\in S^{(m)}$, then both $J^-=J\cap[0,c)$ and $J^+=J\cap (c,1)$ are subsets of connected components of $I\setminus S^{(m)}$. This gives at most two distinct itineraries of order $m$ for points in $J\setminus\{c\}$.  Because $f$ is piecewise increasing  and right continuous, the itinerary of order $m$ of the singular point $c$ coincides with that of $J^+$.

Let $n_0=\lceil \log\tau/\log\lambda\rceil$. Clearly, the length of $f^{n}(W)$ is less than $\tau$, for every connected component $W$ of $I\setminus S^{(n)}$ whenever $n\geq n_0$.

Set $\alpha_n:=\# \mathcal{I}_n$. Therefore,   it follows from the previous claim that $\alpha_{n+m}\leq 2\,\alpha_n$, for every $n\geq n_0$. So $\displaystyle \alpha_{n_0+im}\leq 2^i \alpha_{n_0}$, for every $i\geq 0$. Thus,
$$
\alpha_n\leq  2^{\frac{n-n_0}{m}}\alpha_{n_0},\; \mbox{for every}  \;\, n\geq n_0.
$$
Taking into account the choice of $m$, we get $\alpha_n\leq C \rho^n$, for every $n\geq n_0$, where $C:=2^{-n_0/m}\alpha_{n_0}$. Hence, $\frac1n\log \alpha_n\leq \frac1n\log C+  \log \rho$, which implies that $\limsup_n\frac1n\log \alpha_n\leq \log\rho$.  As $\rho>1$ can be chosen arbitrarily close to $1$, this shows that $h_{\text{sing}}(f)=0$.
\end{proof}

As $f$ has zero singular entropy,  we obtain the following corollary from \cite[Proposition 6.6]{CGMU16}.

\begin{corollary}\label{cor:disco}
The set $\omega(f)$ has Hausdorff dimension zero. In particular, $\omega(f)$ is a totally disconnected set. 
\end{corollary}

\begin{remark} Note that the above claim is trivial whenever $f$ is asymptotically periodic, as in this case $\omega(f)$ is a finite set.
\end{remark}
\begin{remark} In the case $f$ is a contracted rotation with an irrational rotation number~\eqref{eq:contractedrotation}, it is proved in \cite{JO19} that the closure of the limit set $C=\overline{\omega(f)}$ is a Cantor set. Moreover, in the particular case, where the slope of $f$  equals $ \lambda=\frac{1}{n}$,  $n=2,3, \dots$, it is showed in \cite{BKLN20} that every point  $x\in C$, but $x=0$ or $x=1$, is a transcendental number.
\end{remark}

Let $\text{Per}(f)$ denote the set of periodic points of $f$ and $O(f,x)$ denote the forward orbit of $x$ under $f$, i.e., $O(f,x)=\{f^n(x): n\geq 0\}$. We introduce the following equivalence relation $\sim$ on the set of singular points of $f$:
\begin{equation}\label{eqclass}
\mbox{ if $x,y\in S$ we write $x\sim y$ if and only if } \; \omega(f,x) =  \omega(f,y).
\end{equation}
Define the quotient set $A(f):=S/\sim$. 

The following theorem gives an upper bound for the number of periodic orbits of $f$.

\begin{theorem}\label{thm:number}
Assume  $\text{Per}(f)$ is a nonempty set. Then
there is a map $$\psi: \text{Per}(f) \to A(f)$$ with the following property: $\psi (x)=\psi(y)$ if, and only if, $O(f,x)=O(f,y)$, i.e. $x$ and $y$ belong to the same periodic orbit. 
\end{theorem}

\begin{proof}

Let $x\in \text{Per}(f)$ with period $p\in\Nn$. Since $S^{(p)}$ is finite (see Lemma~\ref{lem:Sn2}), we can write $S^{(p)}=\{c_0,c_1,\ldots, c_{m-1}\}$ where $c_0=0<c_1<\cdots<c_{m-1}<1$. Let $c_m=1$ and consider the partition of $I$,
$$
W_0=[c_0,c_1),W_1=[c_1,c_2),\ldots, W_{m-1}=[c_{m-1},c_m).
$$ 
There is a unique $\ell(x)\in\{0,\ldots, m-1\}$ such that $x\in W_{\ell(x)}$. Since $f^p(x)=x$ and $f^p|_{W_{\ell(x)}}$ is increasing and $\lambda^p$-Lipschitz (see Lemma~\ref{lem:Sn2}), we conclude that $f^{p}(W_{\ell(x)})\subset W_{\ell(x)}$ and $\omega(f,z) = O(f,x)$ for every $z\in W_{\ell(x)}$. In particular, $\omega(f,c_{\ell(x)})=O(f,x)$. Because $c_{\ell(x)}$ is a singular point of $f^p$, we may define
$$
n(x):= \min\{q\geq 0\colon f^q(c_{\ell(x)})\in S\}.
$$
Notice that, $n(x) < p$. Recall that $N=\# S$. Let $\kappa(x)\in \{0,\ldots, N-1\}$ be such that $s_{\kappa(x)}= f^{n(x)}(c_{\ell(x)})$. 
Finally, define
$$
\psi: x\in \text{Per}(f) \mapsto [s_{\kappa(x)}]\in A(f),
$$
where [$s_{\kappa(x)}$] means the equivalence class of $s_{\kappa(x)}$ \eqref{eqclass}.  Now it is easy to see that $\psi(x)=\psi(y)$ for two periodic points $x,y\in  \text{Per}(f) $, if and only if $O(f,x)=O(f,y)$. Indeed, suppose that $s_{\kappa(x)}\sim s_{\kappa(y)}$ for two periodic points $x,y\in \text{Per}(f)$. Then,  $\omega(f,f^{n(x)}(c_{\ell(x)}))=\omega(f,f^{n(y)}(c_{\ell(y)}))$ which implies that $$O(f,x)=\omega(f,c_{\ell(x)})=\omega(f,f^{n(x)}(c_{\ell(x)}))=\omega(f,f^{n(y)}(c_{\ell(y)}))=\omega(f,c_{\ell(y)})=O(f,y).$$ 

\end{proof}

\begin{remark}
Theorem~\ref{thm:number} implies that any piecewise increasing contraction $f$ has at most $\# A(f)$ periodic orbits, where $\# A(f)$ denotes the cardinality of the set $A(f)$ which is at most the number of domains of continuity of $f$. 
\end{remark}

\subsection{Proof of Theorem~\ref{th:numbermain}}
The discontinuity points of $f$ whose image  equals zero belong to the equivalence class of zero (in the sense defined in \eqref{eqclass}). This implies that $\# A(f)\leq n+1-\ell$ where $\ell=\#\{x\in S\setminus\{0\}\colon x\sim 0\}$. Therefore, by Theorem~\ref{thm:number}, $f$ has at most $n+1-\ell$ periodic orbits.
\subsection{Proof of Corollary~\ref{cor:circle}}

Identifying the circle $\Rr/\Zz$ with the interval $I$ through the canonical bijection $I\hookrightarrow \Rr \to \Rr/\Zz$, we may conjugate $f$ to an interval map $\tilde{f}\colon I\to I$ which is a piecewise increasing contraction. By further rotating the circle, we assume that $0$ is a discontinuity of $f$.  Hence, $\tilde{f}$ has $k-1$ discontinuity points plus some possible extra discontinuity points whose image under $\tilde{f}$ equals zero. By Theorem~\ref{th:numbermain}, these new discontinuity points only decrease the number of periodic orbits. Therefore, $\tilde{f}$ has at most $k$ periodic orbits, and the same is true for $f$.

\subsection{Proof of Corollary~\ref{cor:nPC}}

The piecewise $\lambda$-affine map $f:I\to I$ is a piecewise increasing contraction. Since $F$ is a $k$-interval piecewise $\lambda$-affine function, $f$ has  $k-1$ discontinuity points plus some extra discontinuity points due to the mod 1. These extra discontinuity points have zero image under $f$ (since $F$ at those points takes an integer value). Therefore, by Theorem~\ref{th:numbermain}, $f$ has at most $k$ periodic orbits.

\section{Invariant quasi-partition}\label{sec:partition}
\noindent
Now we recall the notion of {\it  invariant quasi-partition}   \cite[Definition 2.7]{NPR18}.

\begin{defn} Let $f:I \to I$ be an interval map and $m$ a positive integer. Let $\mathcal{P}=\{J_1,\ldots,J_m\}$ be a collection of $m$ pairwise disjoint open subintervals of the interval $I$. We say that  $\mathcal{P}$ is an {\it invariant quasi-partition of $I$ under $f$}, if it satisfies the following properties:
\begin{enumerate}
\item[(P1)] $\displaystyle I \setminus \bigcup_{i=1}^m J_i$ contains at most finitely many points;
\item[(P2)] For every $\ell=1,\ldots,m$, there is $1\le \tau(\ell)\le m$ such that $f(J_\ell)\subset J_{\tau(\ell)}$.
\end{enumerate}
\end{defn}

Throughout the rest of this section we assume that $f\colon I\to I$ is a piecewise increasing contraction as defined in Section~\ref{sec:intro}.
Below we state a result  \cite[Lemma 2.8]{NPR18} which assures the existence of an invariant quasi-partition for $f$. For the convenience of the reader we include here a proof. Recall that $S$ is the set of singular points of $f$.

\begin{lemma}\label{lem:quasi-partition}
 If the set
\begin{equation}\label{eq:Q}
Q:=\bigcup_{n\geq0}f^{-n}(S)
\end{equation}
is finite, then $f$ admits an  invariant quasi-partition.
\end{lemma}

\begin{proof}
Let $\mathcal{P}=
\{J_\ell\}_{\ell=1}^{m}$ denote the finite collection of all connected components of $I\setminus Q$. Since $Q$ is finite, $I\setminus \bigcup_{i=1}^{m}J_i$ contains at most finitely many points, thus $\mathcal{P}$ verifies (P1).
In order to prove (P2), suppose by contradiction that $f(J_\ell)\cap Q\neq \emptyset$ for some $1\le \ell \le m$. Then $J_\ell \cap f^{-1}(Q)\neq \emptyset$. Using the fact that $f^{-1}(Q)\subset Q$, we conclude that $J_\ell\cap Q\neq\emptyset$, which contradicts the definition of $J_\ell$. 
\end{proof}

In the following we use the notation $\displaystyle f(x^\pm)=\lim_{y\to x^{\pm}}f(y)$.

\begin{defn}\label{def:percon}
A point $x\in S\cup\{1\}$ is called a \textit{left periodic singular point of $f$} if there exists $n\in\Nn$ such that $f^n(x^-)=x$.
\end{defn}

The following theorem is adapted from the results of \cite{NPR18}. 

\begin{theorem}\label{thm:partition}
Let $f$ be a piecewise increasing contraction with no left periodic singular point. If $Q$ is finite, then $f$ is asymptotically periodic.
\end{theorem}

\begin{proof}
Let $\mathcal{P}=
\{J_\ell\}_{\ell=1}^{m}$ denote the invariant quasi-partition of $f$ constructed in Lemma~\ref{lem:quasi-partition}, i.e., $\mathcal{P}$ is the finite collection of the connected components of $I\setminus Q$.
Notice that, since $S\subset Q$ (see \eqref{eq:Q}), for every $\ell\in\{1,\ldots,m\}$ there is $\eta(\ell)\in\{1,\ldots,N\}$ such that  $J_\ell\subset I_{\eta(\ell)}$. 

First we show that $\omega(f,x)$ is a periodic orbit for every $x\in I$. We consider two cases:

\begin{enumerate}
\item[(i)]  If $x\in I\setminus Q$, then there is a  sequence $(\ell_n)_{n=0}^\infty \in \{1,\ldots,m\}^{\Nn_0}$ such that $f^n(x)\in J_{\ell_n}$ and $\ell_{n+1}=\tau(\ell_n)$, for every $n\geq0$. 
Clearly, the sequence $(\ell_n)_{n=0}^\infty$ is eventually periodic, i.e., there exist $k\geq0$ and $p\geq1$ such that $\ell_{k +p}=\ell_k$. Therefore, $f^{p}(J_{\ell_k})\subset J_{\ell_{k+p}}=J_{\ell_k}$. Let $J_{\ell_k}=(c,d)$ where $c,d\in Q\cup\{1\}$. Since $f^p|_{(c,d)}$ is $\lambda^p$-Lipschitz  (Lemma~\ref{lem:Sn2}), and extends to $[c,d]$ by uniform continuity, there exists a unique $z\in [c,d]$ such that $\omega(f^p,x)=\{z\}$.  We want to show that $f^p(z)=z$  which means that $z$ is a periodic point of $f$ and $\omega(f,x)=O(f,z)$. 
We claim that $z\in[c,d)$, which implies that $f^p(z)=z$ because $f^p$ is right continuous by Lemma~\ref{lem:Sn2} (here we use the hypothesis that $f$ is piecewise increasing).  Now we prove the claim. Suppose, by contradiction, that $z=d$. Then $f^p(d^-)=d$. If $d=1$, then $1$ is a left periodic singular point of $f$, thus contradicting the fact that $f$ has no left periodic singular point. So we may suppose that $d<1$. Because $d\in Q$, we  define
$$
r=\min\{n\geq0\colon f^n(d)\in S\}.
$$ 
Notice that $r<p$. Indeed, if $r\geq p$, then $d\notin S^{(p)}$, which means that $f^p$ is continuous at $d$, i.e., $f^p(d)=f^p(d^-)=d$. Hence $O(f,d)\cap S=\emptyset$, contradicting the fact that $d\in Q$. Now take any increasing sequence $x_n\nearrow d$. We know that $f^p(x_n)\to d$. Since $f^r$ is continuous at $d$, we obtain
\begin{align*}
f^r(d)=f^r(f^p(d^-))&=f^r(\lim_{n\to\infty}f^p(x_n))\\
&=\lim_{n\to\infty}f^r(f^p(x_n))\\
&=\lim_{n\to\infty}f^p(f^r(x_n))\\
&=f^p(f^r(d)^-),
\end{align*}
where in the last equality we have used the fact that $f$ is piecewise increasing.
Because $f^r(d)$ is a singular point of $f$, we conclude that $f^r(d)$ is a left periodic singular point of $f$,  which contradicts the hypothesis. Hence, $\omega(f,x)$ is a periodic orbit. 

\item[(ii)] If $x\in Q$, then either $f^{n}(x)\in Q$ for every $n\geq0$, which implies that the orbit of $x$ is finite,  hence eventually periodic, or else there is $k\geq1$ such that $f^k(x)\in I\setminus Q$. In the later case we reduce to the case (i). Therefore, $\omega(f,x)$ is a periodic orbit. 
\end{enumerate}

Finally, as $\mathcal{P}$ is finite, the map $f$ has at most  finitely many periodic orbits. This shows that $f$ is asymptotically periodic.

\end{proof}

\begin{remark}
When  $f\colon I\to I$ is a piecewise contraction not necessarily piecewise increasing and possibly with left periodic singular points, the first part of the proof of Theorem~\ref{thm:partition} shows that if $Q$ is finite,  then for every connected component $J$ of $I\setminus Q$ there is a unique $z\in \overline{J}$ such that $\omega(f,x))=\omega(f,y)$ for every $x,y\in J$ and $\omega(f,x)$ is  finite  for every $x\in J$. This implies that $\omega(f)$ is a finite set provided $Q$ is also finite.  
\end{remark}

\begin{remark}
The left periodic singular point hypothesis in Theorem~\ref{thm:partition}  cannot be removed as the following example shows. 
Let $f:I\to I$ 
be the piecewise increasing contraction,
$$
f(x)=\begin{cases}
\frac{x}{2} +\frac14,&  0\leq x<\frac12\\
\frac{x}{2}-\frac14,&  \frac12\leq x<1\\
\end{cases}.
$$
The map $f$ admits an invariant quasi-partition since $Q = \{0,\frac12\}$. Moreover, as $f(\frac12^-)=\frac12$ the discontinuity is a left periodic singular point of $f$. Therefore, Theorem~\ref{thm:partition} cannot be applied to this case. In fact, $f$ is not asymptotically periodic as the $\omega$-limit set of any point equals $\{\frac12\}$ and the orbit of the discontinuity is not periodic. 
\end{remark}

\section{Piecewise $\lambda$-affine contractions}\label{sec:affine}
\noindent
Given $k\in \Nn$, let $F_\lambda: I \to \Rr$  be the $k$-interval piecewise $\lambda$-affine function defined by the tuples $\veca=(0,a_1,\ldots,a_{k-1})$ and $\vecb=(b_1,\ldots,b_k)$  (see \eqref{eq:FF}). Let $f_\lambda: I \to I$ be the piecewise $\lambda$-affine map defined by $f_\lambda=F_\lambda \pmod{1}$.

Throughout this section  $f_\lambda\colon I\to I$ is a piecewise $\lambda$-affine map 
 as defined above. We will fix the tuples $\veca$ and $\vecb$ and let $\lambda$ vary in $(0,1)$. Notice that $f_\lambda$ is a piecewise increasing contraction as defined in Section~\ref{sec:preliminary}.
 
It is convenient to set $a_0=0$ and $a_k=1$.  Recall, from Definition~\ref{def:Zind}, that the tuples $\veca$ and $\vecb$ are called $\Zz$-independent if and only if
$$
a_{i}-b_j\notin\Zz,\quad\forall\, (i,j)\in\{1,\ldots,k\}^2.
$$
Following the terminology in Section~\ref{sec:preliminary}, we denote the set of singular points of $f_\lambda$ by $S_\lambda$, and, given $n\in\Nn$, we denote the set of singular points of $f_\lambda^n$ by $S_\lambda^{(n)}$. We also denote 
 by $N_\lambda$ the number of connected components of $I\setminus S_\lambda^{(1)}$. Notice that $S_\lambda=S_\lambda^{(1)}$. As in Section~\ref{sec:preliminary}, the set of singular points of $f_\lambda$ defines a collection of intervals $I_j=[s_{j-1},s_j)$, $1\leq j\leq N_\lambda$, which forms a partition of $I$.
On each interval $I_j$, where $1\le j\leq N_\lambda$, according to \eqref{eq:FF} the map $f_\lambda$ takes the expression 
\begin{equation}\label{eq:f}
f_\lambda(x)=\varphi_j(x),\quad \forall x\in I_j,
\end{equation}
with $\varphi_j\colon \Rr\to\Rr$ defined by $\varphi_j(x):=\lambda x + \delta_j$,
where $\displaystyle \delta_j = \beta_j+p_j$  for some  $p_j \in \Zz$ and $\beta_j\in\{b_1,\ldots,b_k\}$.   Notice that $|\delta_j|\leq 1$ for every $1\leq j\leq N_\lambda$.  Similarly to $N_\lambda$, the parameters $\delta_j$ defining $\varphi_j$ may vary with $\lambda$. Indeed, the set of such $\lambda$'s where both $N_\lambda$ and the parameters $\delta_j$ change value are contained in the set
\begin{equation}\label{eq:V}
\mathcal{V}:=\{\lambda\in(0,1)\colon \exists \, 1\le j \le k \; \mbox{ such that} \;  F_\lambda(a_{j}^-) \in\Zz\text{ or } F_\lambda(a_j^+)\in \Zz\}.
\end{equation}
Notice that, $\mathcal{V}$ is a finite set and both $N_\lambda$ and the parameters $\delta_j$ remain constant as $\lambda$ varies inside each connected component of $(0,1)\setminus \mathcal{V}$. 

A crucial step in the proof of Theorem~\ref{th:main} is to estimate  the growth of a larger set of itineraries which contains all itineraries of nearby maps $f_\lambda$, for this purpose we need to exclude possible bifurcations of singular points. 

\begin{defn}\label{def:sing}
A \textit{singular connection of $f_\lambda$ of order $n\geq1$} is an $n$-tuple $(i_0,i_1,\ldots,i_{n-1})\in \{1,\ldots,N_\lambda\}^{n}$ such that
$$
\varphi_{i_{n-1}}\circ\cdots\circ\varphi_{i_0}(\{a_0,a_1,\ldots,a_k\})\cap \{a_0,a_1,\ldots,a_k\} \neq \emptyset.
$$
We say that \textit{$f_\lambda$ has a singular connection} if it has a singular connection of some order $n\geq1$.
\end{defn}

Notice that, if $f_\lambda$ has a left periodic singular point (see Definition~\ref{def:percon}), then it has a singular connection.
\begin{lemma}\label{lem:sing connections}
Let $f_\lambda$ be a family of piecewise $\lambda$-affine maps defined by $\Zz$-independent tuples $\veca$ and $\vecb$.  Then the set
$$
\{\lambda\in(0,1)\colon f_\lambda\text{ has a singular connection}\}
$$
is at most countable.
\end{lemma}

\begin{proof}
Let $J$ be a connected component of $(0,1)\setminus \mathcal{V}$ (see \eqref{eq:V}). Since $\mathcal{V}$ is finite there are at most a finite number of connected components.
As previously discussed, $N_\lambda$ and the parameters $\delta_{j}$ defining $\varphi_j$ in \eqref{eq:f} remain constant for every $\lambda\in J$.

Now, given $\lambda\in J$, $f_\lambda$ has a singular connection of order of $n\geq1$ if there exist  $\omega=(i_0,\ldots,i_{n-1})\in\{1,\ldots,N_\lambda\}^n$ and $x,y\in \{a_0,a_1,\ldots,a_k\}$  such that 
\begin{equation}\label{singular connection}
y=\lambda^n x+\delta_{i_{n-1}}+\delta_{i_{n-2}}\lambda+\cdots+\lambda^{n-1}\delta_{i_0}.
\end{equation}
Since $\veca$ and $\vecb$ are $\Zz$-independent, we have $y\neq \delta_{i_{n-1}}$. Consequently, the polynomial 
$$
Q_{x,y,\omega}(\lambda)= y-\left( \lambda^n x+\delta_{i_{n-1}}+\delta_{i_{n-2}}\lambda+\cdots+\lambda^{n-1}\delta_{i_0}\right)
$$
is not identically zero. Thus, it has at most finitely many roots. Therefore, up to countably many $\lambda$'s in $J$,  equation \eqref{singular connection} does not hold for every $x,y\in \{a_0,a_1,\ldots,a_k\}$ and $\omega$ of any order. This shows that, up to many countable $\lambda$'s in the interval $(0,1)$, $f_\lambda$ has no singular connection.

\end{proof}

Let $J$ be a connected component of $(0,1)\setminus \mathcal{V}$ (see \eqref{eq:V}). Notice that $\mathcal{V}$ is finite. Given $\nu\in J$ and $0<\varepsilon<1$, we define the open interval
\begin{equation}\label{eq:interval}
J_\varepsilon(\nu):=J\cap (\nu-\varepsilon,\nu+\varepsilon).
\end{equation}
Recall that $\mathcal{I}_n$ is the set of itineraries of order $n$ of $f_\lambda$ (see Section~\ref{sec:preliminary}). Of course, it depends on the choice of $\lambda$ and we will write $\I_n(\lambda)$ to stress its dependency. Denote by $\I_n^\varepsilon=\I_n^\varepsilon(\nu)$ the union of all $\I_n(\lambda)$ over $\lambda\in J_\varepsilon(\nu)$.

\begin{lemma}\label{lem:singular entropy}
Let $f_\lambda$ be a family of piecewise $\lambda$-affine maps defined by tuples $\veca$ and $\vecb$. Let $\nu\in J$ and suppose that  $f_\nu$ has no singular connection. Then
$$
 \lim_{\varepsilon\to 0^+}\lim_{n\to\infty}\frac1n\log\# \I_n^\varepsilon(\nu)=0.
$$
\end{lemma}

\begin{proof}
The proof of this lemma is an adaptation of the proof of Proposition~\ref{lem:singent}. 
As in that proof,  given $\rho>1$,  let $m=\lceil\log 2 / \log \rho\rceil$ and $\tau(m,\lambda)>0$  be the smallest distance between any two singular points of $f_\lambda^m$.  Because $f_\nu$ has no singular connection, the sets $S_\nu, f_\nu^{-1}(S_\nu),\ldots, f_\nu^{-m+1}(S_\nu)$ are pairwise disjoint.  Therefore, there is an $\varepsilon_0>0$ such that for every $\lambda\in J_{\varepsilon_0}(\nu)$, the sets $S_\lambda, f_\lambda^{-1}(S_\lambda),\ldots, f_\lambda^{-m+1}(S_\lambda)$ remain pairwise disjoint  and have constant cardinality, i.e.,  $\# f_\lambda^{-i}(S_\lambda)=\# f_\nu^{-i}(S_\nu)$ for $i=0,\ldots,m-1$.  This implies that
$$
\tau(m):=\inf\left\{\tau(m,\lambda)\colon \lambda\in J_{\varepsilon_0}(\nu)\right\}>0.
$$
Hence,  any interval $J$ whose length is less than $\tau(m)$ will intersect at most one singular point of $S_\lambda^{(m)}$ for every $\lambda\in J_{\varepsilon_0}(\nu)$. 

Now we choose $n_0\in\Nn$ large enough,  depending on $m$ and $\varepsilon_0$,  such that for every $n\geq n_0$ and every $\lambda\in J_{\varepsilon_0}(\nu)$, the length of $f^{n}_\lambda(W)$ is less than $\tau(m)$ for every connected component $W$ of $I\setminus S_\lambda^{(n)}$.
The rest of the proof follows the same lines as the proof of Proposition~\ref{lem:singent}. So we conclude that $\lim_n\frac1n\log \# \I_n^\varepsilon\leq \log\rho$ for every $0<\varepsilon\leq \varepsilon_0$.  As $\rho>1$ can be chosen arbitrarily close to $1$, this proves the lemma.
\end{proof}

Given  $n\in\Nn$ and $\omega=(i_0,\ldots,i_{n-1})\in\{1,\ldots,N_\lambda\}^{n}$, we define the polynomial
$$
H_\omega(\lambda):=\varphi_{i_{n-1}}\circ\cdots\circ\varphi_{i_0}(0)=\sum_{j=0}^{n-1}\lambda^j \delta_{i_{n-1-j}},
$$
where both $N_\lambda$ and the parameters $\delta_j$ remain constant for every $\lambda\in J_\varepsilon(\nu)$ (see \eqref{eq:V} and \eqref{eq:interval}).
Notice that $H_\omega(\lambda)+\lambda^nx=f^n_\lambda(x)$ for any $x\in I$, where $\omega$ is the corresponding itinerary of order $n$ of $x$ by $f_\lambda$.
Also define
\begin{equation}\label{eq:Omega}
\Omega_\varepsilon(\lambda):=\bigcap_{m\geq1}\overline{\bigcup_{n\geq m}\bigcup_{\omega\in\I_n^\varepsilon} \{H_\omega(\lambda)} \}.
\end{equation}

\begin{lemma}\label{lem:covering}
Let $f_\lambda$ be a family of piecewise $\lambda$-affine maps defined by tuples $\veca$ and $\vecb$ and let $J$ be a connected component of $(0,1)\setminus \mathcal{V}$.  Assume $
\nu\in J$, $\varepsilon>0$, $\lambda\in J_\varepsilon(\nu)$ as in \eqref{eq:interval} and $n\in\Nn$.  Then the set $\Omega_\varepsilon(\lambda)$ can be covered by finitely many intervals of length $4\lambda^n/(1-\lambda)$ which are centered at the points $H_\omega(\lambda)$, where $\omega\in \I_n^\varepsilon$.
\end{lemma}

\begin{proof}
Given $y\in \Omega_\epsilon(\lambda)$, there are $n_k\nearrow\infty$ and $\omega_k\in\I_{n_k}^\varepsilon$ such that $y_k:=H_{\omega_k}(\lambda)\to y$, as $k\to\infty$. Take $k$ sufficiently large such that $n_k\geq n$ and $|y-y_k|\leq \lambda^{n}$.  Let $\omega_k=(i_0,i_1,\ldots,i_{n_k-1})$ and denote by $[\omega_k]$ the last $n$ entries of $\omega_{k}$.  Clearly, $[\omega_k]\in  \I_n^\varepsilon$ and
$$
\left|H_{\omega_k}(\lambda)-H_{[\omega_{k}]}(\lambda)\right|=\left|\sum_{j=n}^{n_k-1}\lambda^j \delta_{i_{n_k}-1-j}\right|\leq \sum_{j=n}^{n_k-1}\lambda^j \leq \frac{\lambda^n}{1-\lambda}.
$$
Therefore,
\begin{align*}
|y-H_{[\omega_{k}]}(\lambda)|&\leq |y-y_k| + |H_{\omega_k}(\lambda)-H_{[\omega_{k}]}(\lambda)|\\
&\leq  \frac{2\lambda^n}{1-\lambda}.
\end{align*}
\end{proof}

Recall the set $Q$ defined in Lemma~\ref{lem:quasi-partition}. Because $Q$ varies with $\lambda$, we shall denote it by $Q_\lambda$. The next result gives a sufficient condition for the set $Q_\lambda$ to be finite whenever $\lambda\in J_\varepsilon(\nu)$ (see \eqref{eq:interval}). Recall the definition of $\Omega_\varepsilon(\lambda)$ in \eqref{eq:Omega}.


\begin{lemma}\label{lem:suff}
If $$\displaystyle \Omega_\varepsilon(\lambda)\cap \{a_0,a_1,\ldots,a_{k-1}\}=\emptyset,$$ then the set $Q_\lambda$ is finite. 
\end{lemma}

\begin{proof}
Since $\Omega_\varepsilon(\lambda)\cap  \{a_0,a_1,\ldots,a_{k-1}\}=\emptyset$, there exist $n_0\in\Nn$ and $\delta>0$ such that
\begin{equation}\label{deltaclose}
\min_{0\leq i<k}|H_\omega(\lambda)-a_i|\geq \delta,\quad \forall\,n\geq n_0,\,\omega\in\I_n^\varepsilon.
\end{equation}
Let $n\geq n_1:=\max\{n_0, \lceil\log\delta/\log\lambda\rceil\}$ and suppose that there is $x\in I$ such that $f_\lambda^n(x)\in S_\lambda$ but $f_\lambda^m(x)\notin S_\lambda$ for every $0\leq m< n$, i.e., $x$ is a singular point of $f_\lambda^{n+1}$ but not of a lower iterate of $f_\lambda$. We have two cases: 
\begin{enumerate}
\item[(i)]  If $f^n_\lambda(x)\in \{a_0,a_1,\ldots,a_{k-1}\}$, then $|H_\omega(\lambda)-f^n_\lambda(x)|\leq \lambda^n<\delta$, where $\omega$ is the itinerary of order $n$ associated to $x$. Thus, $H_\omega(\lambda)$ is $\delta$-close to $\{a_0,a_1,\ldots,a_{k-1}\}$ which contradicts \eqref{deltaclose}. 
\item[(ii)]  If $f^n_\lambda(x)\in S_\lambda\setminus \{a_0,a_1,\ldots,a_{k-1}\}$, then $f^{n+1}_\lambda(x)=0$. Let $\omega$ be the itinerary of order $n+1$ associated to $x$, then 
$$
|H_\omega(\lambda)|=|H_\omega(\lambda)-f^{n+1}_\lambda(x)|\leq \lambda^{n+1}<\delta,
$$ 
thus $H_\omega(\lambda)$ is $\delta$-close to $a_0=0$ which contradicts \eqref{deltaclose}. 
\end{enumerate}
Both cases contradict  \eqref{deltaclose}. Thus, for $n\geq n_1$ no new singular point of $f_\lambda^n$ is created, i.e., $f^n_\lambda$ and $f_\lambda^{n_1}$ have  the same singular points  for every $n\geq n_1$.
Therefore, the set $\bigcup_{n\geq0}f^{-n}_\lambda(S_\lambda)$ is finite. 
\end{proof}

\section{Proof of Theorem~\ref{th:main} and Corollary~\ref{cor:main}}\label{sec:proof}
\noindent

In this section we prove Theorem~\ref{th:main} and Corollary~\ref{cor:main}. We will use the following metric \L{}ojasiewicz-type inequality \cite[Theorem 4.1 and Theorem 4.6]{G17}.

\begin{lemma}\label{polyestimate}
Let $0<b<1$ and $r\geq 9$. There exist $0<\theta\leq 1$ and $0<\epsilon_0<1$ such that if $0<\epsilon<\epsilon_0$ and $p(x)$ is a polynomial of degree $n\geq \lceil1/\theta\rceil$ of the form
$$
p(x)=1+ c_1x+c_2x^2+\cdots+c_nx^n,\quad c_i\in[-r,r],
$$
then the following holds:
$$
\Leb(\left\{x\in [0,b] \colon |p(x)|<\epsilon \right\})\leq C \epsilon^{\theta},
$$
where $\Leb$ means  Lebesgue measure and
$$
C= \frac{2^{5+\frac3\theta}(1+r)^{\frac2\theta} n^{2+\frac2\theta}(2+\frac1\theta)}{\epsilon_0^2}.
$$
\end{lemma}

\subsection{Proof of Theorem~\ref{th:main}}\label{sec:thm_main}
Let $f_\lambda$ be a family of piecewise $\lambda$-affine maps defined by $\Zz$-independent $k$-tuples $\veca=(0,a_1,\ldots,a_{k-1})$ and $\vecb=(b_1,\ldots,b_k)$.  It is convenient to set $a_0=0$ and $a_k=1$.  We want to show that the set 
$$
\{\lambda\in(0,1)\colon f_\lambda\text{ is not asymptotically periodic}\}
$$
has zero Hausdorff dimension. Denote by $E$ the set of $\lambda\in(0,1)$ such that $f_\lambda$ has a singular connection. By Lemma~\ref{lem:sing connections}, the exceptional set $E$ is at most countable. Notice that, $f_\lambda$ has no left periodic singular point for every $\lambda\in (0,1)\setminus E$. Thus, according to Theorem~\ref{thm:partition}, it is enough to show that 
$$
Z:=\{\lambda\in(0,1)\colon \lambda\notin E,\, Q_\lambda \text{ is not finite}\},
$$
has zero Hausdorff dimension.

Let $J$ be a connected component of $(0,1)\setminus \mathcal{V}$ (see \eqref{eq:V}).  Fix $\nu\in J\setminus E$ and $0<\varepsilon_0<1-\nu$. Given $0<\varepsilon\leq \varepsilon_0$,  define
$$
Z_\varepsilon(\nu):=Z\cap J_\varepsilon(\nu),
$$
where $J_\varepsilon(\nu)$ is the interval defined in \eqref{eq:interval}.  

Our goal is to compute the Hausdorff dimension of the set $Z_\varepsilon(\nu)$. In the next couple of lemmas we will construct a suitable cover of $Z_\varepsilon(\nu)$.

\begin{lemma}\label{lem:1}
For every $0<\varepsilon\leq \varepsilon_0$ and $n\in\Nn$,
$$
Z_\varepsilon(\nu)\subset \bigcup_{\omega\in \I_n^\varepsilon}\bigcup_{j=0}^{k-1}\left\{\lambda\in J_{\varepsilon_0}(\nu)\colon |H_\omega(\lambda)-a_j|\leq  \frac{4(\nu+\varepsilon_0)^n}{1-\nu-\varepsilon_0}\right\}.
$$
\end{lemma}
\begin{proof}
By Lemma~\ref{lem:suff}, 
$$
Z_{\varepsilon}(\nu)\subset X:= \{\lambda\in J_{\varepsilon}(\nu)\colon \Omega_\varepsilon(\lambda)\cap \{a_0,a_1,\ldots,a_{k-1}\}\neq\emptyset\}.
$$
According to Lemma~\ref{lem:covering}, the set $\Omega_\varepsilon(\lambda)$, with $\lambda\in J_\varepsilon(\nu)$, can be covered by $\# \I_n^\varepsilon$ intervals of length $4\lambda^n/(1-\lambda)$ centred at the points $\{H_\omega(\lambda)\colon \omega\in \I_n^\varepsilon\}$.  Therefore,  for every $\lambda \in X$,  there exist $a_j$ with $j\in\{0,\ldots, k-1\}$ and $\omega\in \I_n^\varepsilon$ such that $| H_\omega(\lambda)-a_j|\leq 4\lambda^n/(1-\lambda)\leq 4(\nu+\varepsilon_0)^n/(1-\nu-\varepsilon_0)$.
The conclusion follows by noticing that $J_\varepsilon(\nu)\subset J_{\varepsilon_0}(\nu)$.
\end{proof}

In the following lemma, we use the \L{}ojasiewicz-type inequality of Lemma~\ref{polyestimate} to cover $Z_\varepsilon(\nu)$ using intervals with controlled estimates on their diameter. 
\begin{lemma}\label{lem:2}
There exist $0<\theta\leq 1$, $n_0\geq 2$ and $C>0$ such that the following holds.  For every $0<\varepsilon\leq \varepsilon_0$, $n\geq n_0$,  $\omega\in \I_n^{\varepsilon}$ and $j\in\{0,\ldots,k-1\}$, there exist intervals $U_{\omega,j}^{(1)},\ldots,U_{\omega,j}^{(n-1)}$ each of $\diam(U_{\omega,j}^{(i)})<\eta_n$ such that
$$
\left\{\lambda\in J_{\varepsilon_0}(\nu)\colon |H_\omega(\lambda)-a_j|\leq  \frac{4(\nu+\varepsilon_0)^n}{1-\nu-\varepsilon_0}\right\}\subset \bigcup_{i=1}^{n-1} U_{\omega,j}^{(i)},
$$
where
$$
\eta_n:=C (n-1)^{2+\frac2\theta}(\nu+\varepsilon_0)^{\theta n}.
$$
\end{lemma}
\begin{proof}
Recall that the parameters $\delta_j$ in \eqref{eq:f} remain constant for every $\lambda\in J_{\varepsilon_0}(\nu)$ and $|\delta_j|\leq 1$.  For any $0<\varepsilon\leq \varepsilon_0$, $0\leq j <k$, $n\geq2$ and $\omega\in \I_n^{\varepsilon}$, the polynomial $H_\omega(\lambda)-a_j \in \Rr[\lambda]$ has degree equal to $n-1$. Indeed, because $f_\lambda$ is defined by $\Zz$-independent tuples $\veca$ and $\vecb$, we can write
\begin{align*}
H_\omega(\lambda)-a_j&=\delta_{i_{n-1}}+\delta_{i_{n-2}}\lambda+\cdots+\delta_{i_0}\lambda^{n-1}-a_j\\
&=(\delta_{i_{n-1}}-a_j)\left(1+\delta'_{i_{n-2}}\lambda+\cdots+\delta'_{i_0}\lambda^{n-1}\right),
\end{align*}
where $\omega=(i_0,\ldots, i_{n-1})$, $\delta'_{i_m}:=\delta_{i_m}/(\delta_{i_{n-1}}-a_j)$ for $m=0,\ldots, n-2$ and
$$
0<|\delta'_{i_m}|\leq \frac{1}{|\delta_{i_{n-1}}-a_j|}\leq \frac{1}{\displaystyle\min_{1\leq i,j\leq k}\min_{p\in\Zz}|p+b_i-a_j|}<\infty.
$$ 
Notice that $b_i\notin \Zz$,  and therefore $\delta_i\neq 0$, for any $i\in\{1,\ldots,k\}$, since $a_k=1$ and $\veca$ and $\vecb$ are $\Zz$-independent. Define,
$$
r:=\max\left\{9,\frac{1}{\displaystyle\min_{1\leq i,j\leq k}\min_{p\in\Zz}|p+b_i-a_j|}\right\}.
$$
Taking $n_0\geq2$ sufficiently large, 
we can apply Lemma~\ref{polyestimate} with $\epsilon = \frac{4r(\nu+\varepsilon_0)^n}{1-\nu-\varepsilon_0}$ and get the following estimate which holds for every $n\geq n_0$,
$$
\Leb\left\{\lambda\in J_{\varepsilon_0}(\nu)\colon |H_\omega(\lambda)-a_j|\leq  \frac{4(\nu+\varepsilon_0)^n}{1-\nu-\varepsilon_0} \right\}\leq C (n-1)^{2+\frac2\theta} (\nu+\varepsilon_0)^{\theta n},
$$
where $0<\theta\leq 1$, and $C>0$ is a constant independent of $\omega$, $n$,  $j$ and $\varepsilon$. 
Therefore, since $H_{\omega}$ is a polynomial of degree $n-1$, there exist $n-1$ intervals $U_{\omega,j}^{(1)},\ldots,U_{\omega,j}^{(n-1)}$ of diameter less than $\eta_n:=C (n-1)^{2+\frac2\theta}(\nu+\varepsilon_0)^{\theta n}$ that form the required cover.
\end{proof}

Now we complete the proof of Theorem~\ref{th:main}.

Notice that $\lim_{n\to+\infty}\eta_n=0$ where $\eta_n$ is the sequence in Lemma~\ref{lem:2}.  By Lemma~\ref{lem:1} and Lemma~\ref{lem:2}, for every $0<\varepsilon\leq \varepsilon_0$,  $0<\sigma\leq 1$ and $n\geq n_0$,  we have 
\begin{align*}
\mathcal{H}_{\eta_n}^\sigma (Z_\varepsilon(\nu))&=\inf\left\{\sum_i\diam(U_i)^\sigma\colon Z_\varepsilon(\nu)\subset \bigcup_i U_i,\, \diam(U_i)<\eta_n\right\}\\
&\leq \sum_{\omega\in \I_n^\varepsilon}\sum_{j=0}^{k-1}\sum_{i=1}^{n-1}\diam\left(U_{\omega,j}^{(i)}\right)^\sigma\\
&< \sum_{\omega\in \I_n^\varepsilon}\sum_{j=0}^{k-1}\sum_{i=1}^{n-1} \eta_n^\sigma\\
&= C^\sigma (\#\I_n^\varepsilon)\, k\, (n-1)^{1+\sigma(2+\frac2\theta)}(\nu+\varepsilon_0)^{\sigma\theta n}.
\end{align*}

By Lemma~\ref{lem:singular entropy},  there is $0<\varepsilon\leq \varepsilon_0$ and $n_1\geq n_0$ such that,
$$
\log(\#\I_n^\varepsilon) < -\frac{\sigma\theta n}{2} \log(\nu+\varepsilon_0),\quad \forall\,n\geq n_1.
$$
Hence, 
$
(\#\I_n^\varepsilon)(\nu+\varepsilon_0)^{\sigma\theta n}<(\nu+\varepsilon_0)^\frac{\sigma\theta n}{2}
$ for every $n\geq n_1$, 
which implies that $$\lim_{n\to+\infty}\mathcal{H}_{\eta_n}^\sigma (Z_\varepsilon(\nu)) = 0.$$ Thus $\mathcal{H}^\sigma(Z_\varepsilon(\nu))=0$ for every $0<\sigma\leq 1$, and $Z_\varepsilon(\nu)$ has Hausdorff dimension equal to zero. Because $Z$ is a countable union of the family of sets $\{Z_{\varepsilon_i}(\nu_i)\}_i$ each with zero Hausdorff dimension, we conclude that $Z$ also has zero Hausdorff dimension. \qed

\subsection{Proof of Corollary~\ref{cor:main}}

Let $\lambda$ and $b$ be two numbers such that $0<1-\lambda < b<1$. Let $R_{\lambda, b}$  be the contracted rotation defined by $\lambda$ and $b$  (see \eqref{eq:contractedrotation}).  We will write $R_{\lambda, b}$ as a piecewise $\lambda$-affine map generated by a $k$-interval piecewise $\lambda$-affine function.
Set $\veca=(0,a_1,a_2,\ldots, a_{k-1})\in A^{(k)}$ and  $\vecb=(b,\ldots,b)\in\Rr^k$. Using the $k$-tuples $\veca$ and $\vecb$, we define the $k$-interval piecewise $\lambda$-affine function $F_{\veca,\vecb,\lambda}$ (see \eqref{eq:FF}).  Let $f_{\veca,\vecb,\lambda}$ be the piecewise $\lambda$-affine map associated to $F_{\veca,\vecb,\lambda}$., i.e., $f_{\veca,\vecb,\lambda} = F_{\veca,\vecb,\lambda} \pmod{1}$.  The maps $f_{\veca,\vecb,\lambda}$ and $R_{\lambda, b}$ are identical.
This proves that the exceptional set $\mathfrak{E}^{(k)}$ contains a set which is metrically isomorphic to the exceptional set $A^{(k)}\times \mathcal{E}$ (see \eqref{exceptional1} for the definition of $\mathcal{E}$) and we know that $\dim_H \mathcal{E}=1$ (\cite[Theorem 7.2]{JO19}). This shows that $\dim_H\mathfrak{E}^{(k)} \ge k$.

Now we prove that the exceptional set $\mathfrak{E}^{(k)}$ is a null Lebesgue measure subset of $\Rr^{2k}$.  Denote by $f_{\veca,\vecb,\lambda}$ the piecewise $\lambda$-affine map defined by the parameters $(\veca,\vecb,\lambda)\in A^{(k)}\times \Rr^k\times(0,1)$.  In the following, we identify $A^{(k)}$ with a subset of $\Rr^{k-1}$ by dropping the first coordinate.  Denote by $\Omega$ the set of $(\veca,\vecb)\in A^{(k)}\times\Rr^k$ such that $\veca$ and $\vecb$ are $\Zz$ -- independent.  Cearly, $\Omega$ is a full measure subset of $\Rr^{2k-1}$.  By the proof of Lemma~\ref{lem:sing connections},  the set
$$
E=\{(\veca,\vecb,\lambda)\in \Omega\times(0,1)\colon f_{\veca,\vecb,\lambda}\text{ has a singular connection}\}
$$
is contained in a countable union of co-dimension one algebraic subsets of $\Rr^{2k}$.  Hence,  $E$ is a null subset of $\Rr^{2k}$. 
By Theorem~\ref{thm:partition}, 
$$
\mathfrak{E}^{(k)}\setminus E\subset Z:=\{(\veca,\vecb,\lambda)\in A^{(k)}\times \Rr^k\times(0,1)\colon (\veca,\vecb,\lambda)\notin E,\, Q_{\veca,\vecb,\lambda}\text{ is not finite}\},
$$
where $Q_{\veca,\vecb,\lambda}$ is defined in Lemma~\ref{lem:quasi-partition}.   The set $Z$ is measurable because $Q_{\veca,\vecb,\lambda}$ is a countable union of the sets $f^{-n}_{\veca,\vecb,\lambda}(\{a_j\})$ with $n\in\Nn$ and $j=0,\ldots,k-1$.  Moreover,  we have proved in Theorem~\ref{th:main} that the sections $Z_{\veca,\vecb}=\{\lambda\in(0,1)\colon (\veca,\vecb,\lambda)\in Z\}$ are null sets for every $(\veca,\vecb)\in\Omega$.  By Fubini theorem, we conclude that $Z$ is a null set,  thus $\mathfrak{E}^{(k)}$ is also a null set.\qed

\section*{Acknowledgement}
The author J.P.G. was partially supported by the Project PTDC/MAT-PUR/29126/2017 and by the Project CEMAPRE/REM - UIDB /05069/2020 - financed by FCT/MCTES through national funds. The author A.N. graciously acknowledges the support  of CEFIPRA through the Project No. 5801-1/2017. The project leading to this publication has received funding from Excellence Initiative of Aix-Marseille University - A*MIDEX and Excellence Laboratory Archimedes LabEx (ANR-11-LABX-0033), French "Investissements d'Avenir" programmes.

\bibliographystyle{amsplain}

\end{document}